\documentclass{amsart}
\usepackage{fullpage}
\usepackage{amssymb}
\usepackage{amsmath}
\usepackage{amsthm}
\usepackage{url}
\usepackage{hyperref}
\usepackage{mathabx,rotating}
\usepackage{multirow}
\usepackage{mathabx}
\usepackage{tablefootnote}
\usepackage{pdflscape}
\usepackage{graphicx}

\usepackage[numbers]{natbib}

\newtheorem{theorem}{Theorem}

\newtheorem{lemma}{Lemma}

\newtheorem{definition}{Definition}

\newcommand{\cal}{\mathcal}

\begin{document}

\title{Diagrams of opposition: an algebraic viewpoint}
\date{July 31, 2024}
\author{Chai Wah Wu\\IBM T. J. Watson Research Center\\P. O. Box 218, Yorktown Heights, New York 10598, USA\\e-mail: cwwu@us.ibm.com}

\begin{abstract}
We study the square of opposition and its various geometric generalizations from an algebraic viewpoint. In particular, we show how the various shapes of oppositions can be framed under an algebraic framework and we illustrate this approach with algebraic structures beyond the traditional logical structures. 
\end{abstract}


\maketitle

\section{Introduction}
The classical (Aristotelian) square of opposition is a diagram describing the relationships between propositions in the study of syllogisms \cite{Englebretsen1976}.
The vertices of this square are the $4$ logical statements: $A$: ``All $P$ is $Q$'', $E$: ``All $P$ is not $Q$'', $I$: ``Some $P$ is $Q$'', and $O$: ``Some $P$ is not $Q$''. 
This diagram can be extended to a graded square of opposition where the truth value of each vertex can take on a range of values beyond just true and false.

The square of opposition has also been extended to hexagons \cite{blanche:vrin:1966}, cubes \cite{dubois:cube:2015} and other types of geometric shapes \cite{Dubois2020}. 
In most of these discussions, the objects on the vertices of these geometric shapes are derived from logical structures, such as Boolean algebra of 2 symbols, Boolean algebra of subsets, T-norms, etc.
The purpose of this paper is to generalize these logical structures to more general algebraic structures. We provide some specific examples of these structures that goes beyond the traditional logical structures considered.

\section{Square of opposition on an algebraic structure}

As is common in the classical square of opposition, the following terminologies are used:
\begin{definition} \label{def:logic}
Let $A$ and $B$ be two logical propositions.
\begin{itemize}
\item $A$ is {\em contrary} to $B$ if they cannot both be true.
\item $A$ is {\em sub-contrary} to $B$ if they cannot both be false.
\item $A$ is {\em contradictory} to $B$ if $A$ is true if and only if $B$ is false.
\item $A$ {\em sub-implicates} $B$ if $A$ implies $B$ (written as $A\Rightarrow B$), i.e. either $A$ is false or $B$ is true.
\item $A$ {\em super-implicates} $B$ if $B$ implies $A$, i.e. either $A$ is true or $B$ is false.
\end{itemize}
\end{definition}

Let us define the following algebraic structure:

\begin{definition}
A tuple $(X,\neg, \preceq)$ is said to be in class $\cal C$ if the following conditions are satisfied:
\begin{itemize}
\item $\neg: X \rightarrow X$ is an involutory unary operation on the set $X$, i.e., $\neg (\neg (x)) = x$ for all $x\in X$,
\item $\preceq$ is a partial order\footnote{I.e., $\preceq$ is antisymmetric, reflexive and transitive. Also known as a weak, reflexive \cite{Wallis2012} or non-strict \cite{Simovici2014} partial order.} on $X$,
\item $a\preceq b$ if and only if $\neg b\preceq \neg a$,
\item There exists an element $0\in X$ such that $\neg 0 \preceq 0$. The nonempty set $Z$ (also referred to as {\em zeros}) is defined as $Z = \{x\in X:\neg x \preceq x\}$.
\end{itemize}
\label{def:classc}
\end{definition}

We use $P\prec Q$ to denote that $P\preceq Q$ and $P\neq Q$. 
We next define the following $4$ logical statements using the same letters $A$, $E$, $I$ and $O$ that are used in the classical square of opposition above.
\begin{itemize}
\item $A$: $P\preceq Q$
\item $E$: $P\preceq \neg Q$
\item $I$: $P\not\preceq \neg Q$
\item $O$: $P\not\preceq  Q$
\end{itemize}

The following result shows that these statements satisfy the relationships in a square of opposition as illustrated in Figure \ref{fig:square}. 

\begin{theorem}
Consider $(X,\neg, \preceq) \in \cal C$ with $P,Q\in X$. Suppose there exists $0\in Z$ such that $P\succ 0$. Then
\begin{itemize}
\item $A\Rightarrow I$.
\item $E\Rightarrow O$.
\item $A$ and $E$ cannot both be true.
\item $I$ and $O$ cannot both be false.
\item $A$ is equivalent to the negation of $O$.
\item $E$ is equivalent to the negation of $I$.
\end{itemize}
\label{thm:square}
\end{theorem}
\begin{proof} It is clear that $A$ is true if and only if $O$ is false and that $E$ is true if and only if $I$ is false. Suppose $A$ is true. Then $0\prec P\preceq Q$ implies that $\neg Q \preceq \neg P\prec \neg 0 \preceq 0$. Thus $\neg Q \prec P$ and therefore $P\not\preceq \neg Q$ by antisymmetry of $\preceq$.
A similar argument shows that $E\Rightarrow O$. 
Suppose both $A$ and $E$ are true. Then $E$ implies $Q\preceq \neg P$ and combined with $A$ implies $0 \prec P\preceq -P$. Since $\neg P\preceq \neg 0\preceq 0$ this results in a contradiction.
Thus $A$ and $E$ cannot both be true.
Similarly this shows that $I$ and $O$ cannot both be false.
\end{proof}

\begin{figure}[htbp]
\centerline{\includegraphics[width=5in]{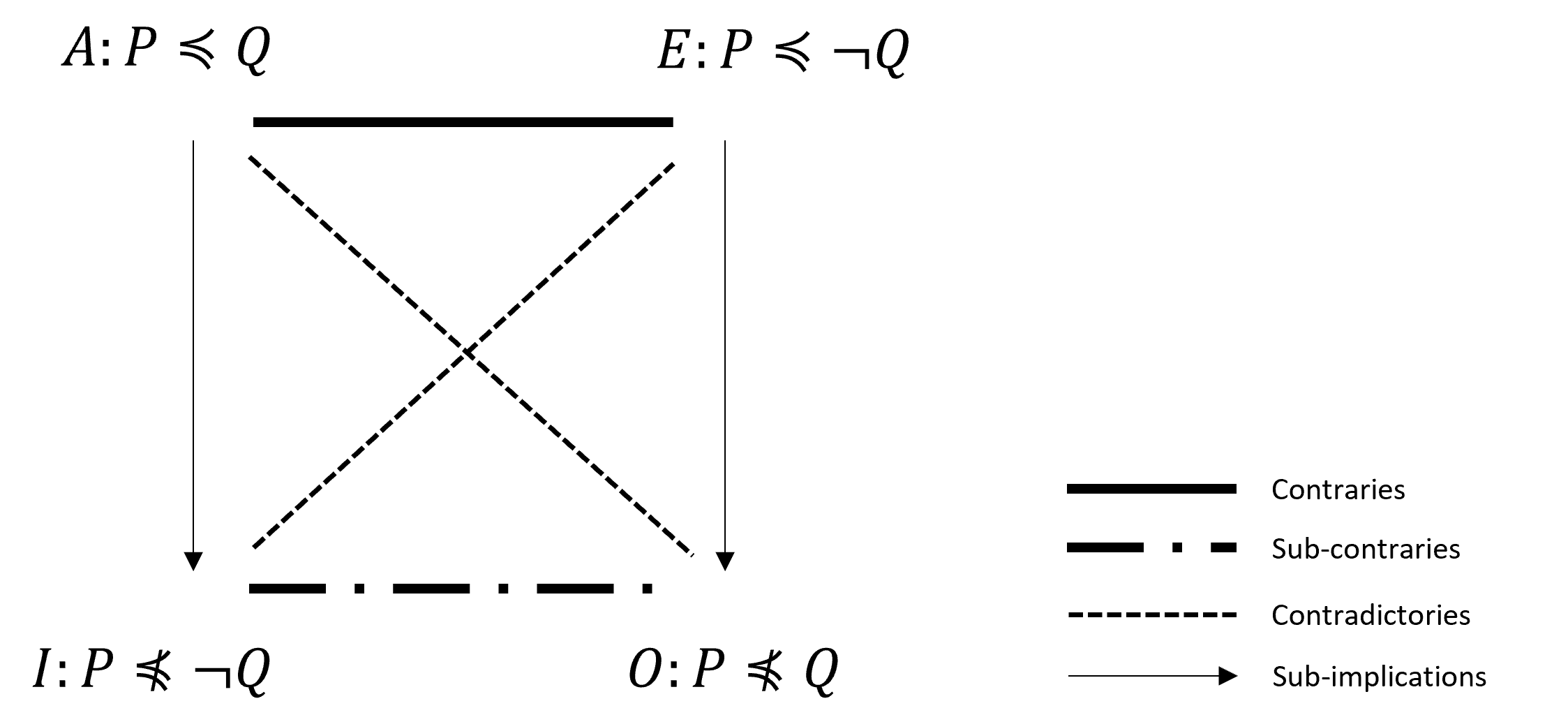}}
\caption{Square of opposition for statements $A$, $E$, $I$ and $O$ of class $\cal C$.}
\label{fig:square}
\end{figure}

Since the formalization of symbolic logic by Boole and others, a modern square has emerged that differs from the classical square.
What mainly distinguishes the classical square from the modern square of Boole is the (implicit) assumption of existential import in the classical square, where the universal quantifier implies the existence of a particular \cite{MACCOLL1905,RUSSELL1905,Wu1969}, i.e. ``All $P$ is $Q$'' implies the existence of a $P$ that is $Q$, i.e. $A$ implies $I$ in the classical square, but not in the modern square\footnote{See \cite{Wreen:1984,westerstaahl2012classical} for other viewpoints on the difference between classical and modern squares of opposition.}. Our formulation and Theorem \ref{thm:square} show that the (controversial) existential import in the classical square can be replaced by the assumption that $P\succ 0$.

\section{Cube of opposition}
Following \cite{Dubois2012}, we define an additional set of $4$ logical statements by substituting $P$ with $\neg P$ and $Q$ with $\neg Q$.
\begin{itemize}
\item $a$: $\neg P\preceq \neg Q$
\item $e$: $\neg P\preceq Q$
\item $i$: $\neg P\not\preceq Q$
\item $o$: $\neg P\not\preceq  \neg Q$
\end{itemize}

We have the following conclusions which form the cube of opposition as illustrated in Figure \ref{fig:cube}. Note that in contrast with \cite{Dubois2012}, there are  links between any two vertices in Figure \ref{fig:cube}. In particular, there are links between $A$ and $a$ and $E$ and $e$ \footnote{These links were also defined in Refs. \cite{Reichenbach1952,Dubois2020}}.

\begin{theorem}
Consider $(X,\neg, \preceq) \in \cal C$ with $P,Q\in X$. Suppose there exists $0\in Z$ such that $P\succ 0$. Then
\begin{itemize}
\item $A\Rightarrow e$.
\item $E\Rightarrow a$.
\item $o\Rightarrow I$.
\item $i\Rightarrow O$.
\item $a$ and $e$ cannot both be true.
\item $i$ and $o$ cannot both be false.
\item $A$ and $E$ cannot both be true.
\item $I$ and $O$ cannot both be false.
\item $a$ and $I$ cannot both be false.
\item $e$ and $O$ cannot both be false.
\item $A$ and $i$ cannot both be true.
\item $E$ and $o$ cannot both be true.
\item $a$ is equivalent to the negation of $o$.
\item $e$ is equivalent to the negation of $i$.
\item $A$ is equivalent to the negation of $O$.
\item $E$ is equivalent to the negation of $I$.
\end{itemize}
If in addition $P\neq Q$ and $P\neq \neg Q$, then
\begin{itemize}
\item $A$ and $a$ cannot both be true.
\item $E$ and $e$ cannot both be true.
\item $I$ and $i$ cannot both be false.
\item $O$ and $o$ cannot both be false.
\item $A\Rightarrow o$.
\item $E\Rightarrow i$.
\item $a\Rightarrow O$.
\item $e\Rightarrow I$.
\end{itemize}
\end{theorem}

\begin{proof}
Suppose $P\succ 0$.
The links on the front and rear faces of the cube follow from Theorem \ref{thm:square}.
$P\succ 0$ implies that $\neg P\prec \neg 0 \preceq 0 \prec P$ and thus $A\Rightarrow e$. Similarly $E\Rightarrow a$. This also implies that $o\Rightarrow I$ and $i\Rightarrow O$ via contraposition.
Since $e$ is equivalent to the negation of $i$, $A\Rightarrow e$ means that $A$ and $i$ cannot both be true.
Since the negation of $I$ implies $a$, $I$ and $a$ cannot both be false. Similarly $e$ and $O$ cannot both be false and $E$ and $o$ cannot both be true. 

Suppose in addition that $P\neq Q$.
If $A$ and $a$ are both true, then $P\preceq Q$ and $Q\preceq P$, i.e. $P=Q$ reaching a contradiction.
Similarly if $P\neq \neg Q$, then if $E$ and $e$ are both true, we get $P=\neg Q$ which is a contradiction.
This also implies that $I$ and $i$ cannot both be false and $O$ and $o$ cannot both be false.
Since $A$ and $a$ cannot both be true, $A$ and the negation of $o$ cannot both be true and this means that $A\Rightarrow o$.
Similar arguments show that $E\Rightarrow i$, $a\Rightarrow O$, and $e\Rightarrow I$.
\end{proof}

Since there is a link between every pair of vertices in Fig. \ref{fig:cube}, a (nontraditional) square of opposition can be constructed from any $4$ distinct vertices of the cube.

\begin{figure}[htbp]
\centerline{\includegraphics[width=5in]{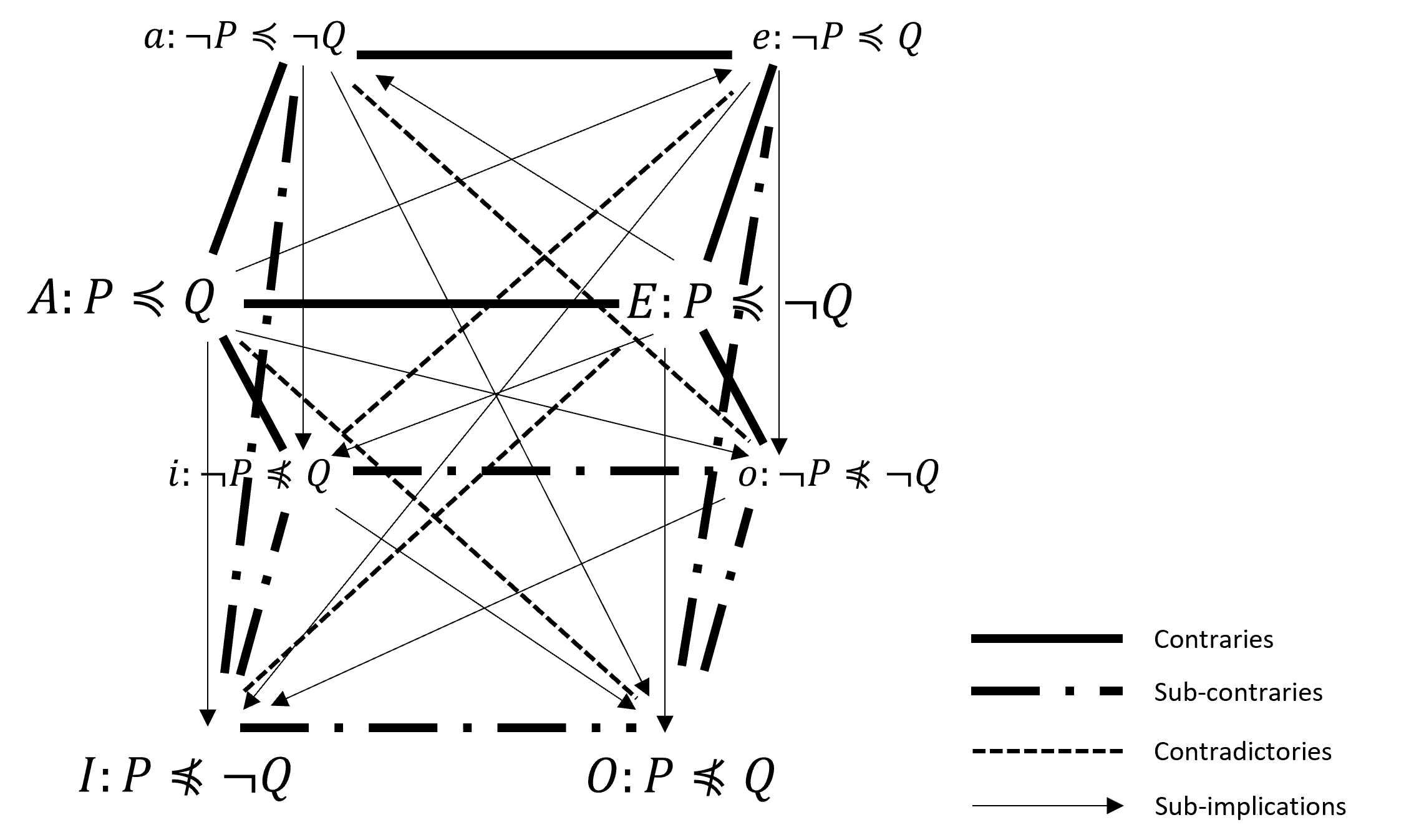}}
\caption{Cube of opposition for statements $A$, $E$, $I$, $O$, $a$, $e$, $i$, and $o$ of class $\cal C$.}
\label{fig:cube}
\end{figure}

\section{Hexagon of opposition}
In \cite{blanche:vrin:1966}, a hexagon of opposition was considered by the addition of two logical statements:
$U = A\vee E$ and $Y = I\wedge O$. It is easy to show that 
\begin{itemize}
\item $A\Rightarrow U$ and $E\Rightarrow U$.
\item $Y\Rightarrow I$ and $Y\Rightarrow O$.
\item $A$ and $Y$ cannot both be true.
\item $E$ and $Y$ cannot both be true.
\item $I$ and $U$ cannot both be false.
\item $O$ and $U$ cannot both be false.
\end{itemize}
This is illustrated in Fig. \ref{fig:hexagon}.

\begin{figure}[htbp]
\centerline{\includegraphics[width=4in]{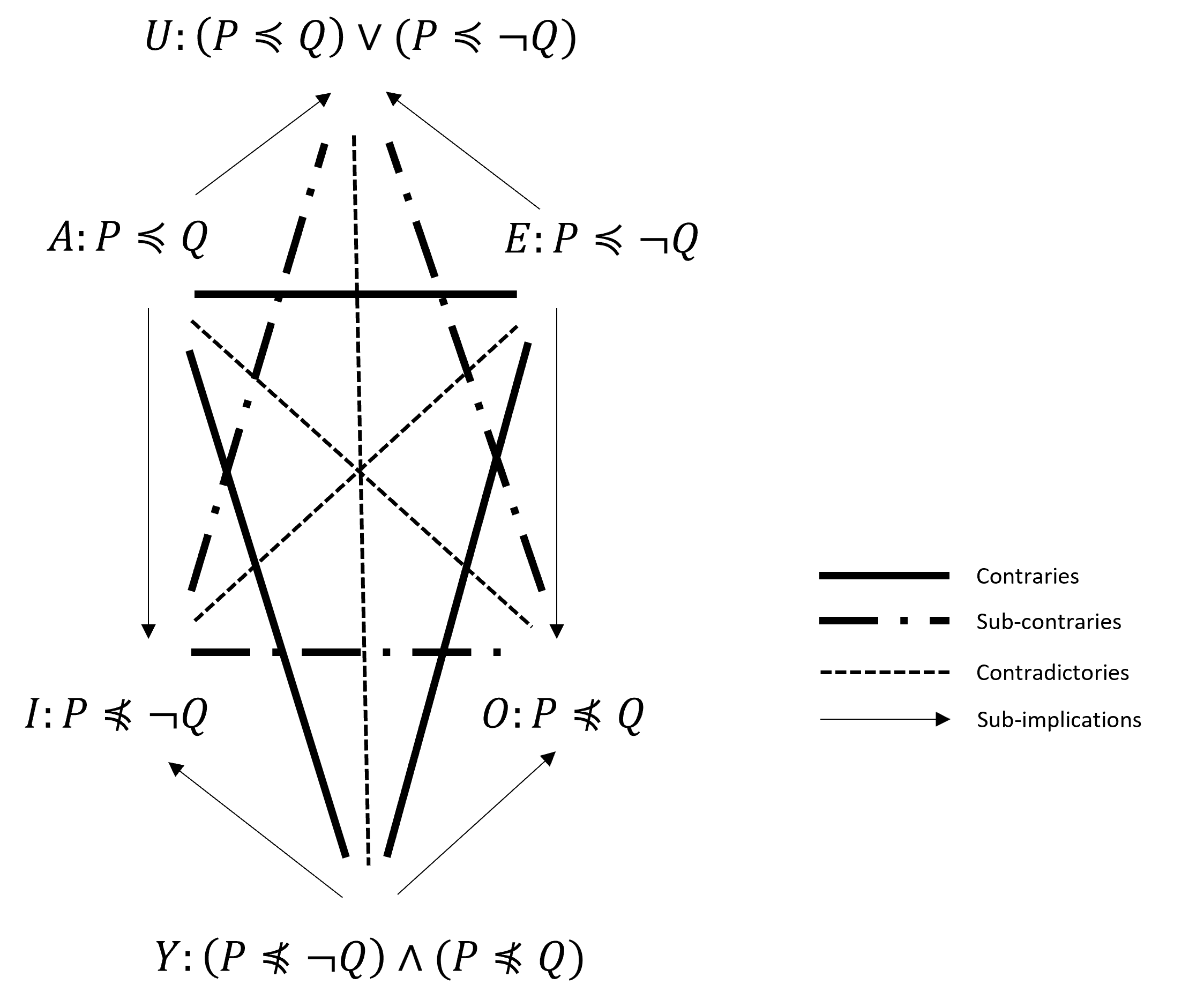}}
\caption{Hexagon of opposition for statements $A$, $E$, $I$, $O$, $U$, $Y$ of class $\cal C$.}
\label{fig:hexagon}
\end{figure}

We define $\max(Y)$ as $\{y\in Y | \forall z\in Y, z\preceq y\}$ and $\min(Y)$ as $\{y\in Y | \forall z\in Y, z\succeq y\}$. 

\begin{lemma}
If either $Q\preceq \neg Q$ or $\neg Q\preceq Q$, then $\max(\{Q,\neg Q\})$ exists and is unique and we denote it 
as $|Q| = \max(\{Q,\neg Q\})$. This also implies that $\min(\{Q,\neg Q\})$ exists and is unique and that $\neg |Q| = \min(\{Q,\neg Q\})$.
\end{lemma}

We can use $|Q|$ to reduce the compound logical statements $U$ and $Y$ as shown in Fig. \ref{fig:hexagon3} by redefining them as: $U:P\preceq |Q|$ and $Y:P\not\preceq |Q|$ .
This follows from the fact that if $|Q|$ exists then $(P\preceq Q)\vee(P\preceq\neg Q)$ if and only if $P\preceq |Q|$. Similarly $(P\not\preceq Q)\wedge(P\not\preceq\neg Q)$ if and only if $P\not\preceq |Q|$
 
\begin{theorem}
Consider $(X,\neg, \preceq) \in \cal C$ with $P,Q\in X$. Suppose either $Q\preceq \neg Q$ or $\neg Q\preceq Q$ and there exists $0\in Z$ such that $P\succ 0$. Then Fig. \ref{fig:hexagon3} is a hexagon of opposition satisfying the corresponding relationships, i.e.
\begin{itemize}
\item $A\Rightarrow I$.
\item $E\Rightarrow O$.
\item $A$ and $E$ cannot both be true.
\item $I$ and $O$ cannot both be false.
\item $A$ is equivalent to the negation of $O$.
\item $E$ is equivalent to the negation of $I$.
\item $A\Rightarrow U$ and $E\Rightarrow U$.
\item $Y\Rightarrow I$ and $Y\Rightarrow O$.
\item $A$ and $Y$ cannot both be true.
\item $E$ and $Y$ cannot both be true.
\item $I$ and $U$ cannot both be false.
\item $O$ and $U$ cannot both be false.
\end{itemize}
\end{theorem}

\begin{figure}[htbp]
\centerline{\includegraphics[width=4in]{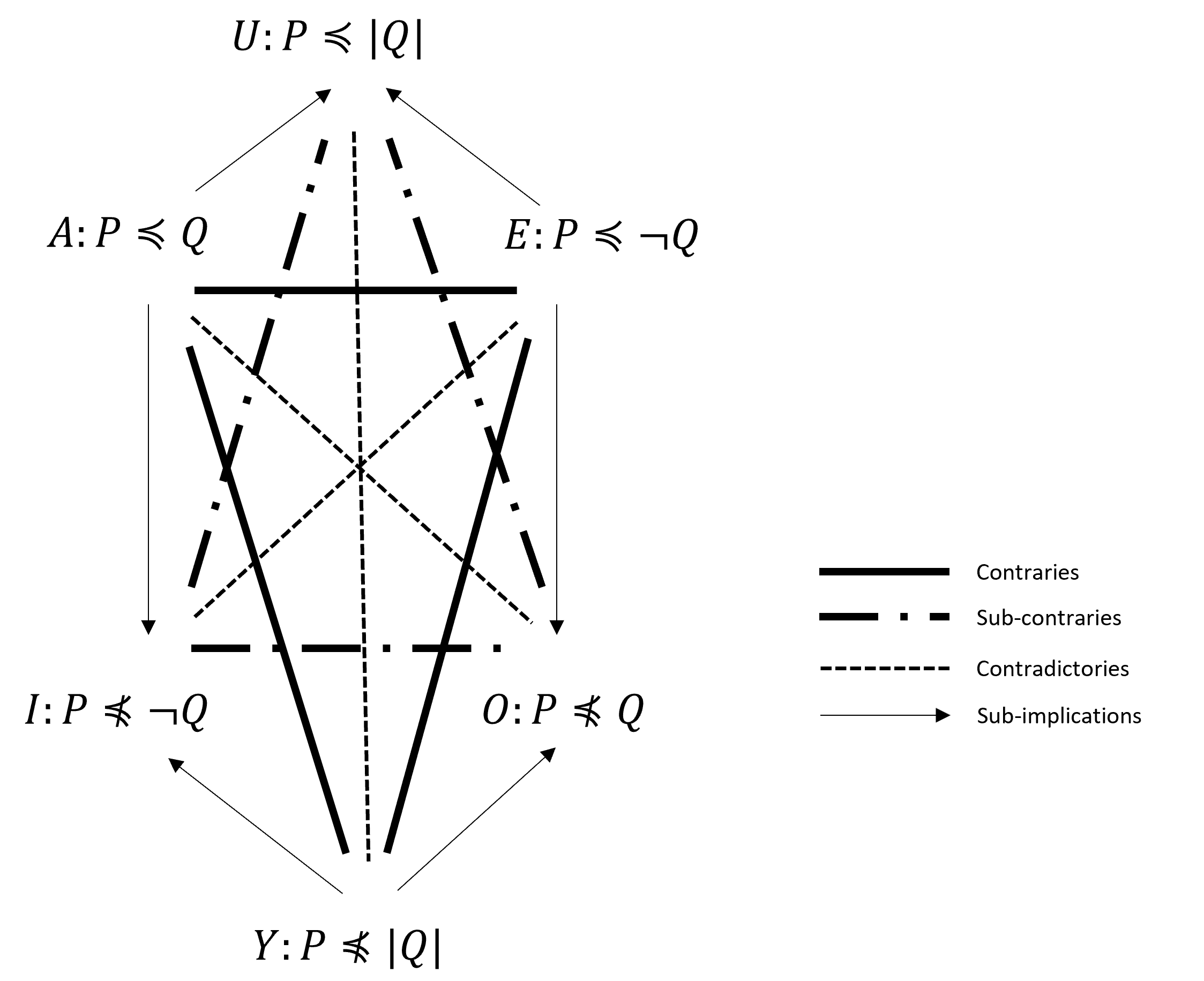}}
\caption{Hexagon of opposition for statements $A$, $E$, $I$, $O$, $U$, $Y$ of class $\cal C$.}
\label{fig:hexagon3}
\end{figure}

\section{Examples of class $\cal C$}
In this section, we show how several of the different types of structures that exhibit a geometric diagram of opposition can be formulated under this algebraic approach. In particular, we show that these structures are examples of class $\cal C$.

\subsection{Augmented Boolean logic} \label{sec:boolean_logic}
In this case we define $X = \left\{0,\frac{1}{2},1\right\}$ with the usual total order $\leq$, $\neg{x} = 1-x$ and $Z = \left\{\frac{1}{2},1\right\}$. In this case, true and false are associated with the values $0$ and $1$ respectvely. In particular, we denote $T_A\in X$ as the numerical value of the truth value of $A$. If $P,Q\in \{0,1\}$, then the inequality $P\preceq Q$ corresponds to the logical implication $P\Rightarrow Q$ and we have recovered the classical diagram of opposition restricted to a single element in the quantifiers which is true even when $P\succ \frac{1}{2}$ is not satisfied.

\subsection{Signed multisets} \label{sec:SM}
Boole was one of the first to systematically link logic with algebra. As was pointed out by Hailperin \cite{Hailperin1981,hailperin:boole:1986}, Boole's algebra (as contrasted with Boolean algebra) is a commutative ring that can be interpreted as a signed multiset, i.e. a multiset where the multiplicities of the elements can be both positive, $0$ or negative. This algebra was studied in \cite{Blizard1990} as a model of logic that contains ZFC and is relatively consistent. 

In  these studies,
the multiplicities are integers, and we consider here a generalization where the multiplicities are themselves elements of a structure in class $\cal C$.
In particular, let $M$ be a set of multiplicities in class $\cal C$ with partial order $\preceq_m$ and negation $\neg_m$ and zero $0_m$.

We denote a (generalized) signed multiset using a Python style dictionary notation as $A = \{a:m_a, b:m_b, c:m_c\}$ indicating that the signed multiset $A$ has elements $a$ with multiplicity $m_a$, $b$ with multiplicty $m_b$ and $c$ with multiplicity $m_c$. 
$\neg A$ is called the complement of $A$ and is defined as the signed multiset where all the multiplicities are negated, i.e. if  $A = \{a:m_a, b:m_b, c:m_c\}$ then  $A = \{a:\neg_m m_a, b:\neg_m m_b, c:\neg_m m_c\}$

The empty set $\emptyset_A$ (which is not unique) is the set containing all elements of $A$ where all the multiplicities are $0_m$. For $A$ and $B$ signed multisets, 
we define $A\preceq B$ if for $b \in B$ with multiplicity $r$, either $b\not\in A$ and $r\succeq 0_m$ or $b\in A$ and $k\preceq_m r$ where $k$ is the multiplicity of $b$ in $A$. 
This implies that $\neg \emptyset_A \preceq \emptyset_B$ for multisets $A$ and $B$. 
If $(M,\oplus)$ is a group with an unit equal to $0_m$ and inverse equal to $\neg_m$, then 
we define the additive union  $A\uplus B$ of two signed multisets $A$ and $B$ as the signed multiset obtained by summing the multiplicities. We define $\neg A$ is the unique multiset such that $A\uplus \neg A = \emptyset_A$.

Note that in \cite{Blizard1990} empty sets do not have proper subsets. Note that in our formulation this is not necessarily true since if $\emptyset_A \preceq A$, then $\neg A \preceq \emptyset_A$.
With this formulation, we obtained a class of signed multisets with a corresponding partial order that is in class $\cal C$ with zero set $Z=\{\emptyset_A\}_A$.
The condition $P\succ 0$ corresponds to a multiset with an element $x$ whose multiplicity $T_x$ satisfies $T_x\succ_m 0_m$. This can be interpreted as a set with an element of positive multiplicity and generalizes the meaning of existential import.

\subsection{Classical sets}
The classical Boolean algebra of subsets \cite{whitesitt:boolean:1961} can be considered a variation of signed multisets.
Let $U$ be the universal set of all elements and $M= \{-1,0,1\}$.
Then we can define the complement of a classical set $\overline{A}$ as $U\uplus \neg A$. Conversely,
we have $\neg A = \overline{A}\uplus \neg U$.
Note that the conditions $P\neq Q$ and $P \neq \neg Q$, which resulted in the links between $A$, $E$, $I$, $O$ and $a$, $e$, $i$, $o$ respectively in Fig. \ref{fig:cube}, correspond to analogous conditions in \cite{Reichenbach1952} for the class of subsets.

\subsection{Vectors}
Given a tuple  $(X,\neg, \preceq)$ in class $\cal C$ we can define a partial order on the Cartesian product $X^n$ of vectors such that $P\preceq Q$ if and only if $P_i \preceq Q_i$ for each corresponding component of $P$ and $Q$ and $\neg P$ is obtained by applying $\neg$ to each component of $P$. It is easy to show that this structure on $X^n$ remain in class $\cal C$. 

An important example is in propositional logic, where a statement with $n$ logical variables corresponds to a truth table of $2^n$ entries and can be represented with a binary vector of length $2^n$.
The equivalence between truth values in Boolean algebra and an algebra in class $\cal C$ allows the relationships of contraries, sub-contraries, contradictories and sub-implications in the square of opposition in Definition \ref{def:logic} to be replaced with a partial order in class $\cal C$ as well. 
Let the vertices of the diagrams of opposition be logical statements on a set $X$. Then to each statement $A$ corresponds a binary vector of truth-values in $\{0,1\}^X$.

\begin{lemma} \label{lem:logic_poset}
Let $T_A$ and $T_B$ denote the truth values vectors of $A$ and $B$ respectively. Let $\vec{1}$ and $\vec{0}$ denote the vectors of all $1$'s and all $0$'s.
\begin{itemize}
\item $A$ is {\em contrary} to $B$ if and only if $T_A+T_B\preceq \vec{1}$.
\item $A$ is {\em sub-contrary} to $B$ if and only if $T_A+T_B\succeq \vec{1}$.
\item $A$ is {\em contradictory} to $B$ if and only if $T_A = \vec{1}-T_B$
\item $A$ {\em sub-implicates} $B$ if and only if $T_A\preceq T_B$
\item $A$ {\em super-implicates} $B$ if and only if $T_A\succeq T_B$ .
\end{itemize}
\end{lemma}
\begin{proof}
A contrary between $A$ and $B$ means that $A$ and $B$ cannot both be true, i.e. each row of the truth table cannot contains 2 true values and this can be represented as $T_A+T_B\preceq \vec{1}$. Similarly a sub-contrary of $A$ and $B$ means that each row of the truth table cannot be both false and this is equivalent to $T_A+T_B\succeq \vec{1}$, a contradictory between $A$ and $B$ corresponds to $T_A = \vec{1}-T_B$ and $A$ sub-implicating $B$ means for each row of the truth table, if $A$ is true, then $B$ is true and this is equivalent to $T_A\preceq T_B$. 
\end{proof}

In \cite{Williamson1972}, the classical square of opposition is compared with the square of opposition using propositional logic. We show that this connection can also be considered under this framework.

The inequality $P\preceq Q$ corresponds to the logical implication $P\Rightarrow Q$. We next show that we still recover the classical diagram of opposition as we move from the first order logic statements in Fig. \ref{fig:square} to propositional logic except for the interesting phenomena where the upper vertices ($A$ and $E$) are now swapped with the lower vertices ($I$ and $O$).
In particular, $P\not\preceq \neg Q$ corresponds to $P\wedge Q$ and by looking at the truth table it is easy to see that $T_{P\wedge Q}\preceq T_{P\Rightarrow Q}$.
Similarly  $T_{P\wedge \neg Q}\preceq T_{P\Rightarrow \neg Q}$. Secondly, $T_{I}+T_{O}\preceq \vec{1}$ and $T_{A}+T_{E}\succeq \vec{1}$ and thus Lemma \ref{lem:logic_poset} shows that we have a square of opposition as illustrated in Fig: \ref{fig:square_prop_logic}.

\begin{figure}[htbp]
\centerline{\includegraphics[width=4in]{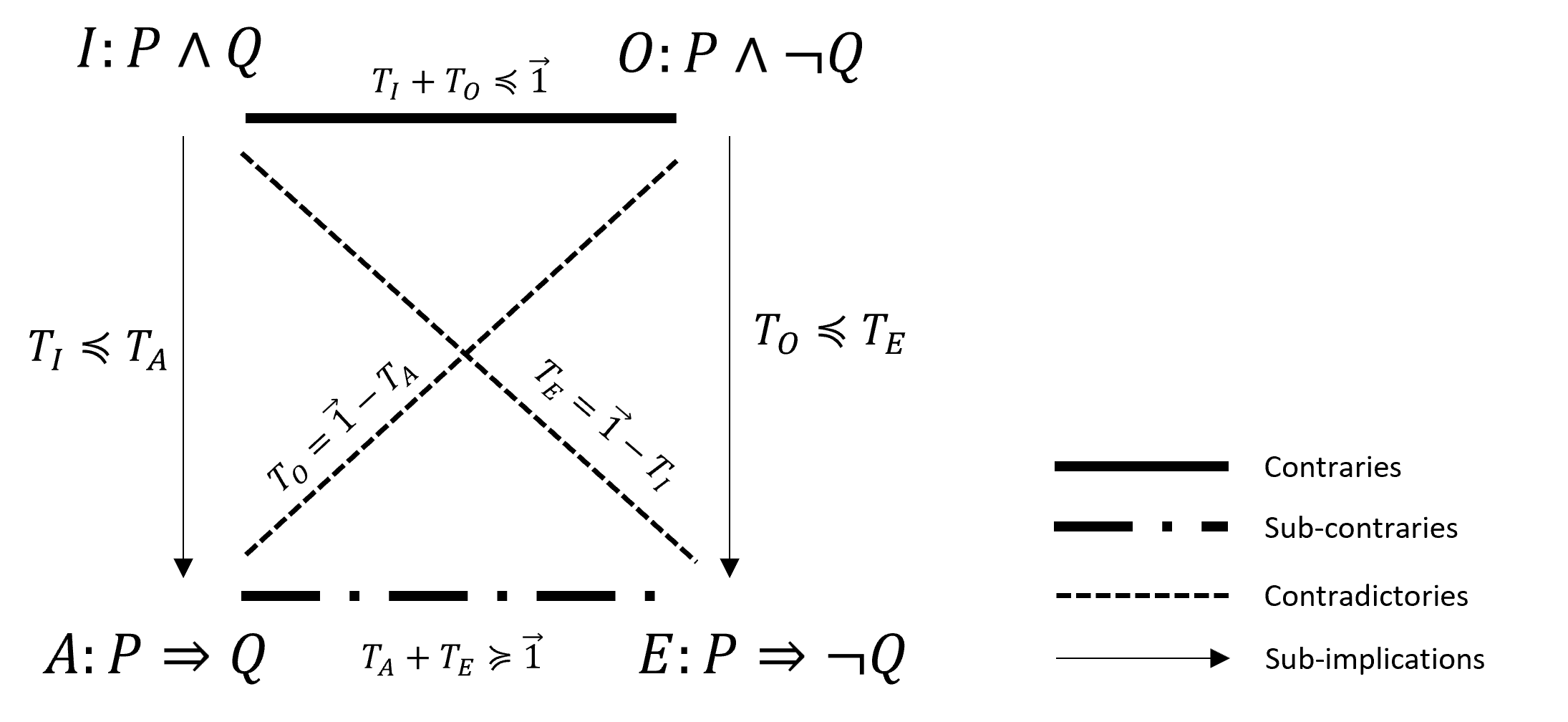}}
\caption{Square of opposition for propositional statements $A$, $E$, $I$, $O$ of class $\cal C$.}
\label{fig:square_prop_logic}
\end{figure}

\subsection{Square matrices}
For the class of complex square matrices, we define the partial ordering $A\preceq B$ if $A$ and $B$ are of the same size and $(B-A)+(B-A)^H$ is positive semidefinite. In this case $0$ corresponds to the zero matrix and $Z$ is the set of matrices $A$ such that $A+A^H = 0$.
Note that here we use the notation $\prec$ as related to the partial ordering of matrices and not as it is usually used for denoting positive definiteness, i.e. $A\prec B$ if $A\preceq B$ and $A\neq B$. This is weaker than the definition of positive definiteness, i.e. $A \succ 0$ means $A$ is positive semidefinite and not equal to the zero matrix, but $A$ is not necessarily positive definite.
In this case the resulting structure is in the class $\cal C$ with the only zero being the all zeros matrix $0$ and $\neg A = -A$.
\subsection{Negation in multi-valued logic}
Consider a strong negation function in multi-valued logic, i.e. a strictly monotone involutive function $\neg:[0,1]\rightarrow [0,1]$ such that $\neg 0 = 1$ and $\neg 1 = 0$. The function $\neg$ can be decomposed as $\phi^{-1}(1-\phi(x))$ for some continous strictly increasing bijection $\phi$ on $[0,1]$ \cite{trillas:negation:1979}. It is clear that $Z$ is nonempty and
$([0,1],\neg,\preceq)$ is in class $C$.

\section{Conclusions}
We show how the square of opposition and its generalization can be applied to various algebraic structures. This continues the study to identify the linkage between logic and algebra, a connection that started with Boole's and Venn's work \cite{boole:logic:1847,venn1880} and continued by many others \cite{Stone1936,Halmos1998} ever since. A recent study connecting multi-valued logic and certain combinatorial and algebraic properties can be found in \cite{Wu:rearrange:2023}.


\begin{thebibliography}{24}
\expandafter\ifx\csname natexlab\endcsname\relax\def\natexlab#1{#1}\fi
\def\docolon{:}
\def\eatcomma#1{}
\def\onlyone#1{\gdef\oneletter{#1}}
\def\sphref#1#2{{\let\#=\docolon\xdef\one{#1}}\href{\one}{#2}}
\def\zhref#1,#2{{\let\#=\docolon\xdef\one{#1}}\href{\one}{#2}}
\expandafter\ifx\csname url\endcsname\relax
  \def\url#1{{\tt #1}}\fi
\newcommand{\enquote}[2]{``#1,''}

\bibitem[Blanch\'{e}(1966)]{blanche:vrin:1966}
Blanch\'{e}, R.,
\newblock {\em Structures Intellectuelles: Essai sur l'Organisation
  Syst\'{e}matique des Concepts},
\newblock Librairie Philosophique J. Vrin, 1966.\eatcomma.

\bibitem[Blizard(1990)]{Blizard1990}
Blizard, W.~D.,
\newblock \enquote{Negative membership},
\newblock {\em Notre Dame Journal of Formal Logic}, vol.~31 (1990).\eatcomma.

\bibitem[Boole(1847)]{boole:logic:1847}
Boole, G.,
\newblock {\em The Mathematical Analysis of Logic, Being an Essay towards a
  Calculus of Deductive Reasoning},
\newblock Macmillan, Barclay, \& Macmillan, 1847.\eatcomma.

\bibitem[Dubois and Prade(2012)]{Dubois2012}
Dubois, D. \unskip, and  H.~Prade,
\newblock \enquote{From {B}lanché’s hexagonal organization of concepts to
  formal concept analysis and possibility theory},
\newblock {\em Logica Universalis}, vol.~6 (2012), pp.~149--169.\eatcomma.

\bibitem[Dubois et~al.(2015)Dubois, Prade, and Rico]{dubois:cube:2015}
Dubois, D., H.~Prade \unskip, and  A.~Rico,
\newblock \enquote{The cube of opposition-a structure underlying many knowledge
  representation formalisms},
\newblock in {\em Proceedings of the Twenty-Fourth International Joint
  Conference on Artificial Intelligence (IJCAI)}, pp.~2933--2939, 2015.

\bibitem[Dubois et~al.(2020)Dubois, Prade, and Rico]{Dubois2020}
Dubois, D., H.~Prade \unskip, and  A.~Rico,
\newblock \enquote{Structures of opposition and comparisons: Boolean and
  gradual cases},
\newblock {\em Logica Universalis}, vol.~14 (2020), pp.~115--149.\eatcomma.

\bibitem[Englebretsen(1976)]{Englebretsen1976}
Englebretsen, G.,
\newblock \enquote{The square of opposition},
\newblock {\em Notre Dame Journal of Formal Logic}, vol.~17 (1976).\eatcomma.

\bibitem[Hailperin(1986)]{hailperin:boole:1986}
Hailperin, T.,
\newblock {\em Boole's Logic and Probability: A Critical Exposition from the
  Standpoint of Contemporary Algebra, Logic and Probability Theory}, 2nd
  edition,
\newblock Elsevier, 1986.\eatcomma.

\bibitem[Hailperin(1981)]{Hailperin1981}
Hailperin, T.,
\newblock \enquote{Boole's algebra isn't {B}oolean algebra},
\newblock {\em Mathematics Magazine}, vol.~54 (1981), p.~173.\eatcomma.

\bibitem[Halmos and Givant(1998)]{Halmos1998}
Halmos, P. \unskip, and  S.~Givant,
\newblock {\em Logic as Algebra}, volume~21 of {\em Dolciani Mathematical
  Expositions},
\newblock The Mathematical Association of America, 1998.\eatcomma.

\bibitem[MacColl(1905)]{MACCOLL1905}
MacColl, H.,
\newblock \enquote{Existential import},
\newblock {\em Mind}, vol.~XIV (1905), pp.~295--296.\eatcomma.

\bibitem[Reichenbach(1952)]{Reichenbach1952}
Reichenbach, H.,
\newblock \enquote{The syllogism revised},
\newblock {\em Philosophy of Science}, vol.~19 (1952), pp.~1--16.\eatcomma.

\bibitem[Russell and MacColl(1905)]{RUSSELL1905}
Russell, B. \unskip, and  H.~MacColl,
\newblock \enquote{The existential import of propositions},
\newblock {\em Mind}, vol.~14 (1905), pp.~398--402.\eatcomma.

\bibitem[Simovici and Djeraba(2014)]{Simovici2014}
Simovici, D.~A. \unskip, and  C.~Djeraba,
\newblock {\em Mathematical Tools for Data Mining: Set Theory, Partial Orders,
  Combinatorics}, chapter~Partially Ordered Sets, pp.~67--95,
\newblock Springer London, 2014.

\bibitem[Stone(1936)]{Stone1936}
Stone, M.~H.,
\newblock \enquote{The theory of representation for {B}oolean algebras},
\newblock {\em Transactions of the American Mathematical Society}, vol.~40
  (1936), pp.~37--111.\eatcomma.

\bibitem[Trillas(1979)]{trillas:negation:1979}
Trillas, E.,
\newblock \enquote{Sobre functiones de negaci\'{o}n en la teor\'{i}a de
  conjuntas difusos},
\newblock {\em Stochastica}, vol.~3 (1979), pp.~47--60.\eatcomma.

\bibitem[Venn(1880)]{venn1880}
Venn, J.,
\newblock \enquote{On the diagrammatic and mechanical representation of
  propositions and reasonings},
\newblock {\em The London, Edinburgh, and Dublin Philosophical Magazine and
  Journal of Science}, vol.~10 (1880), pp.~1--18.\eatcomma.

\bibitem[Wallis(2012)]{Wallis2012}
Wallis, W.~D.,
\newblock {\em A Beginner's Guide to Discrete Mathematics},
\newblock Birkhäuser Boston, 2012.\eatcomma.

\bibitem[Westerst{\aa}hl(2012)]{westerstaahl2012classical}
Westerst{\aa}hl, D.,
\newblock \enquote{Classical vs. modern squares of opposition, and beyond},
\newblock {\em The square of opposition. A general framework for cognition},
  (2012), pp.~195--229.\eatcomma.

\bibitem[Whitesitt(1961)]{whitesitt:boolean:1961}
Whitesitt, J.~E.,
\newblock {\em Boolean Algebra and Its Applications},
\newblock Addison-Wesley Publishing Company, 1961.\eatcomma.

\bibitem[Williamson(1972)]{Williamson1972}
Williamson, C.,
\newblock \enquote{Squares of opposition: Comparisons between syllogistic and
  propositional logic.},
\newblock {\em Notre Dame Journal of Formal Logic}, vol.~13 (1972).\eatcomma.

\bibitem[Wreen(1984)]{Wreen:1984}
Wreen, M.,
\newblock \enquote{Existential import},
\newblock {\em Cr\'{i}tica: Revista Hispanoamericana de Filosof\'{i}a}, vol.~16
  (1984), pp.~59--64.\eatcomma.

\bibitem[Wu(2023)]{Wu:rearrange:2023}
Wu, C.~W.,
\newblock \enquote{On rearrangement inequalities for triangular norms and
  co-norms in multi-valued logic},
\newblock {\em Logica Universalis}, vol.~17 (2023), pp.~331--346.\eatcomma.

\bibitem[Wu(1969)]{Wu1969}
Wu, J.~S.,
\newblock \enquote{The problem of existental import (from {G}eorge {B}oole to
  {P}. {F}. {S}trawson.)},
\newblock {\em Notre Dame Journal of Formal Logic}, vol.~10 (1969).\eatcomma.

\end{thebibliography}
\end{document}